\documentclass[12pt]{article} 
\usepackage{amsfonts,amsmath,latexsym,amssymb,mathrsfs,amsthm}
\usepackage{slashbox}
\usepackage{caption}

\let\OLDthebibliography\thebibliography
\renewcommand\thebibliography[1]{
  \OLDthebibliography{#1}
  \setlength{\parskip}{3pt}
  \setlength{\itemsep}{0pt plus 0.3ex}
}

\def\numberlikeadb{\global\def\theequation{\thesection.\arabic{equation}}}
\numberlikeadb
\newtheorem{theorem}{Theorem}[section]

\newtheorem{corollary}[theorem]{Corollary}

\newtheorem{remark}[theorem]{Remark}

\usepackage{lscape}
\usepackage{caption}
\usepackage{multirow}
\begin{document}

\title{Inequalities for integrals of the modified Struve function of the first kind}
\author{Robert E. Gaunt\footnote{School of Mathematics, The University of Manchester, Manchester M13 9PL, UK}}

\date{April 11, 2018} 
\maketitle

\vspace{-5mm}

\begin{abstract}Simple inequalities for some integrals involving the modified Struve function of the first kind $\mathbf{L}_{\nu}(x)$ are established.  In most cases, these inequalities have best possible constant.  We also deduce a tight double inequality, involving the modified Struve function $\mathbf{L}_{\nu}(x)$, for a generalized hypergeometric function.
\end{abstract}



\noindent{{{\bf{AMS 2010 Subject Classification:}}} Primary 33C10; 26D15

\section{Introduction}\label{intro}

In the recent papers \cite{gaunt ineq1} and \cite{gaunt ineq3}, simple lower and upper bounds, involving the modified Bessel function of the first kind $I_\nu(x)$, were obtained for the integrals
\begin{equation}\label{intbes}\int_0^x \mathrm{e}^{-\gamma t} t^\nu I_\nu(t)\,\mathrm{d}t, \qquad \int_0^x \mathrm{e}^{-\gamma t} t^{\nu+1} I_\nu(t)\,\mathrm{d}t,
\end{equation}
where $x>0$, $0\leq\gamma<1$ and $\nu>-\frac{1}{2}$.  For $\gamma\not=0$ there does not exist simple closed form expressions for the integrals in (\ref{intbes}) The inequalities of \cite{gaunt ineq1,gaunt ineq3} were needed in the development of Stein's method \cite{stein,chen,np12} for variance-gamma approximation \cite{eichelsbacher, gaunt vg, gaunt vg2}, although as they are simple and surprisingly accurate the inequalities may also prove useful in other problems involving modified Bessel functions; see for example, \cite{baricz3} in which inequalities for modified Bessel functions of the first kind were used to obtain lower and upper bounds for integrals involving modified Bessel functions of the first kind.

In this note, we consider the natural problem of obtaining inequalities, involving the modified Struve function of the first kind, for the integrals
\begin{equation}\label{intstruve}\int_0^x \mathrm{e}^{-\gamma t} t^\nu \mathbf{L}_\nu(t)\,\mathrm{d}t, \qquad \int_0^x \mathrm{e}^{-\gamma t} t^{\nu+1} \mathbf{L}_\nu(t)\,\mathrm{d}t,
\end{equation}
where $x>0$, $0\leq\gamma<1$ and $\nu>-\frac{3}{2}$, and $\mathbf{L}_\nu(x)$ is the modified Struve function of the first kind defined, for $x\in\mathbb{R}$ and $\nu\in\mathbb{R}$, by
\begin{equation*}\mathbf{L}_\nu(x)=\sum_{k=0}^\infty \frac{\big(\frac{1}{2}x\big)^{\nu+2k+1}}{\Gamma(k+\frac{3}{2})\Gamma(k+\nu+\frac{3}{2})}.
\end{equation*}
The modified Struve function $\mathbf{L}_\nu(x)$ is closely related to the modified Bessel function $I_\nu(x)$, and either shares or has a close analogue to the properties of $I_\nu(x)$ that were exploited in derivations of the inqualities for the integrals in (\ref{intbes}) by \cite{gaunt ineq1,gaunt ineq3}.  The function $\mathbf{L}_\nu(x)$ is itself a widely used special function; see a standard reference, such as \cite{olver}, for its basic properties.  It arises in manyfold applications, including leakage inductance in transformer windings \cite{hw94}, perturbation approximations of lee waves in a stratified flow \cite{mh69}, scattering of plane waves by soft obstacles \cite{s84}; see \cite{bp13} for a list of further application areas. 



When $\gamma=0$ both integrals in (\ref{intstruve}) can be evaluated exactly.  Indeed, the second integral is equal to $x^{\nu+1}\mathbf{L}_{\nu+1}(x)$ (see \cite{olver}, formula 11.4.29).  The first integral can be evaluated because the modified Struve function $\mathbf{L}_{\nu}(x)$ can be represented as a generalized hypergeometric function.  To see this, recall that the generalized hypergeometric function (see \cite{olver} for this definition and further properties) is defined by
\begin{equation*}{}_pF_q\big(a_1,\ldots,a_p;b_1,\ldots,b_q;x\big)=\sum_{k=0}^\infty \frac{(a_1)_k\cdots(a_p)_k}{(b_1)_k\cdots(b_q)_k}\frac{x^k}{k!},
\end{equation*}
and the Pochhammer symbol is given by $(a)_0=1$ and $(a)_k=a(a+1)(a+2)\cdots(a+k-1)$, $k\geq1$.  Then, for $-\nu-\frac{3}{2}\notin\mathbb{N}$, we have the representation
\begin{equation*}\mathbf{L}_\nu(x)=\frac{x^{\nu+1}}{\sqrt{\pi}2^\nu\Gamma(\nu+\frac{3}{2})} {}_1F_2\bigg(1;\frac{3}{2},\nu+\frac{3}{2};\frac{x^2}{4}\bigg)
\end{equation*}
(see also \cite{bp13} for other representations in terms of the generalized hypergeometric function). A straightforward calculation then yields
\begin{equation}\label{besint6}\int_0^x t^\nu \mathbf{L}_\nu(t)\,\mathrm{d}t=\frac{x^{2\nu+2}}{\sqrt{\pi}2^{\nu+1}(\nu+1)\Gamma(\nu+\frac{3}{2})}{}_2F_3\bigg(1,\nu+1;\frac{3}{2},\nu+\frac{3}{2},\nu+2;\frac{x^2}{4}\bigg).
\end{equation}
When $\gamma\not=0$, there does, however, not exist a closed form formula for the integrals in (\ref{intstruve}).  Moreover, even when $\gamma=0$ the first integral is given in terms of the generalized hypergeometric function.  This provides the motivation for establishing simple bounds, involving the modified Struve function $\mathbf{L}_\nu(x)$, for these integrals.


The approach taken in this note to bound the integrals in (\ref{intstruve}) is similar to that used by \cite{gaunt ineq1,gaunt ineq3} to bound the related integrals involving the modfied Bessel function $I_\nu(x)$, and the inequalities obtained in this note are of a similar form to those obtained for the integrals involving $I_\nu(x)$.  As already noted, the reason for this similarity is because many of the properties of the modified Bessel function $I_\nu(x)$ that were exploited in the proofs of  \cite{gaunt ineq1,gaunt ineq3} are shared by the modified Struve function $\mathbf{L}_\nu(x)$, which we now list.  All these formulas can be found in \cite{olver}, except for the inequality which is given in \cite{bp14}.  Further inequalities for $\mathbf{L}_\nu(x)$ can be found in \cite{bp14} and \cite{bps17}, some of which improve results of \cite{jn98}.

For positive values of $x$ the function $\mathbf{L}_{\nu}(x)$ is positive for $\nu>-\frac{3}{2}$ . The function $\mathbf{L}_{\nu}(x)$ satisfies the recurrence relation and differentiation formula
\begin{align}\label{Iidentity}\mathbf{L}_{\nu -1} (x)- \mathbf{L}_{\nu +1} (x) &= \frac{2\nu}{x} \mathbf{L}_{\nu} (x)+\frac{\big(\frac{1}{2}x\big)^\nu}{\sqrt{\pi}\Gamma(\nu+\frac{3}{2})}, \\
\label{diffone}\frac{\mathrm{d}}{\mathrm{d}x} \big(x^{\nu} \mathbf{L}_{\nu} (x) \big) &= x^{\nu} \mathbf{L}_{\nu -1} (x).
\end{align}
The function $\mathbf{L}_\nu(x)$ has the following asymptotic properties:
\begin{align}\label{Itend0}\mathbf{L}_{\nu}(x)&\sim \frac{2}{\sqrt{\pi}\Gamma(\nu+\frac{3}{2})}\bigg(\frac{x}{2}\bigg)^{\nu+1}, \quad x \downarrow 0, \: \nu>-\tfrac{3}{2}, \\
\label{Itendinfinity}\mathbf{L}_{\nu}(x)&\sim \frac{\mathrm{e}^{x}}{\sqrt{2\pi x}}, \quad x \rightarrow\infty, \: \nu\in\mathbb{R}.
\end{align}
Let $x > 0$. Then 
\begin{equation}\label{Imon}\mathbf{L}_{\nu} (x) < \mathbf{L}_{\nu - 1} (x), \quad \nu \geq \tfrac{1}{2}.  
\end{equation} 

We end this introduction by noting that \cite{gaunt ineq1,gaunt ineq3} also derived lower and upper bounds for the integrals $\int_x^\infty \mathrm{e}^{\beta t} t^\nu K_\nu(t)\,\mathrm{d}t$ and $\int_x^\infty \mathrm{e}^{\beta t} t^{\nu+1} K_\nu(t)\,\mathrm{d}t$, where $x>0$, $\nu>-\frac{1}{2}$, $0\leq\beta<1$ and $K_\nu(x)$ is a modified Bessel function of the second kind.  Analogously to the problem studied in this note it is natural to ask for bounds for the integrals $\int_x^\infty \mathrm{e}^{\beta t} t^\nu \mathbf{M}_\nu(t)\,\mathrm{d}t$ and $\int_x^\infty \mathrm{e}^{\beta t} t^{\nu+1} \mathbf{M}_\nu(t)\,\mathrm{d}t$, where $\mathbf{M}_\nu(x)=\mathbf{L}_\nu(x)-I_\nu(x)$ is the modified Struve function of the second kind.  However, the inequalities of \cite{gaunt ineq1, gaunt ineq3} for integrals involving $K_\nu(x)$ do not have a natural analogue for $\mathbf{M}_\nu(x)$.

Unlike the function $\mathbf{L}_\nu(x)$, for general values of $\nu$, some of the crucial properties of $K_\nu(x)$ that were exploited in the proofs of \cite{gaunt ineq1,gaunt ineq3} do not have an analogue for $\mathbf{M}_\nu(x)$.  Indeed, the function $\mathbf{M}_\nu(x)$ does not have the exponential decay as $x\rightarrow\infty$ that $K_\nu(x)$ has, and is in fact unbounded when $\nu>1$ (see formula 11.6.2 of \cite{olver}).  Moreover, despite possessing some interesting monotonicity properties (see \cite{bp142}), $\mathbf{M}_\nu(x)$ does not have an analogue of the inequality $K_{\nu}(x)>K_{\nu-1}(x)$, $\nu>\frac{1}{2}$, (see \cite{ifantis}) which was heavily used in the proofs of \cite{gaunt ineq1,gaunt ineq3}.  


\section{Inequalities for integrals of the modified Struve function of the first kind}\label{sec2}

In the following theorem, we establish inequalities for the integrals in (\ref{intstruve}), which are natural analogues of the inequalities for the integrals in (\ref{intbes}) that are given in Theorem 2.1 of \cite{gaunt ineq1} and Theorem 2.3 of \cite{gaunt ineq3}. 

\begin{theorem} \label{tiger} Let $n>-1$ and $0\leq \gamma <1$. Then, for all $x>0$,
\begin{align}\label{besi11}\int_0^x \mathrm{e}^{-\gamma t}t^\nu \mathbf{L}_{\nu+n}(t)\,\mathrm{d}t&>\mathrm{e}^{-\gamma x}x^\nu \mathbf{L}_{\nu+n+1}(x), \quad\nu>-\tfrac{1}{2}(n+2),  \\
\label{100fcp}\int_0^xt^{\nu}\mathbf{L}_{\nu}(t)\,\mathrm{d}t &<x^{\nu}\mathbf{L}_{\nu}(x), \quad \nu\geq \tfrac{1}{2}, 
\end{align}
\begin{align}
\label{besi22}\int_0^x t^\nu \mathbf{L}_{\nu+n}(t)\,\mathrm{d}t&<\frac{x^\nu}{2\nu+n+1}\bigg(2(\nu+n+1)\mathbf{L}_{\nu+n+1}(x)-(n+1)\mathbf{L}_{\nu+n+3}(x)\nonumber \\
& \quad -\frac{(n+1)x^{\nu+n+2}}{\sqrt{\pi}2^{\nu+n+1}(2\nu+n+2)\Gamma(\nu+n+\frac{5}{2})}\bigg), \quad\!\! \nu>-\tfrac{1}{2}(n+1), \\
\label{pron}\int_0^x\mathrm{e}^{-\gamma t} t^{\nu}\mathbf{L}_{\nu}(t)\,\mathrm{d}t &\leq \frac{\mathrm{e}^{-\gamma x}}{1-\gamma}\int_0^xt^{\nu}\mathbf{L}_{\nu}(t)\,\mathrm{d}t, \quad \nu \geq \tfrac{1}{2},  \\
\label{besi33}\int_0^x \mathrm{e}^{-\gamma t}t^{\nu}\mathbf{L}_\nu(t)\,\mathrm{d}t&<\frac{\mathrm{e}^{-\gamma x}x^\nu}{(2\nu+1)(1-\gamma)}\bigg(2(\nu+1)\mathbf{L}_{\nu+1}(x)-\mathbf{L}_{\nu+3}(x)\nonumber \\
&\quad\quad\quad\quad\quad\quad\quad\quad\quad-\frac{x^{\nu+2}}{\sqrt{\pi}2^{\nu+2}(\nu+1)\Gamma(\nu+\frac{5}{2})}\bigg), \quad \nu\geq\tfrac{1}{2}, \\
\label{besi44}\int_0^x \mathrm{e}^{-\gamma t}t^{\nu+1} \mathbf{L}_{\nu}(t)\,\mathrm{d}t&\geq\mathrm{e}^{-\gamma x}x^{\nu+1} \mathbf{L}_{\nu+1}(x), \quad \nu>-\tfrac{3}{2}, \\
\label{besi55}\int_0^x \mathrm{e}^{-\gamma t}t^{\nu+1} \mathbf{L}_{\nu}(t)\,\mathrm{d}t&<\frac{1}{1-\gamma}\mathrm{e}^{-\gamma x}x^{\nu+1} \mathbf{L}_{\nu+1}(x), \quad\nu>-\tfrac{1}{2}.
\end{align}
We have equality in (\ref{pron}) and (\ref{besi44}) if and only if $\gamma=0$.  The constants in the bounds (\ref{100fcp})--(\ref{besi55}) cannot be improved, and the constant in (\ref{besi11}) is also best possible if $\gamma=0$.  Inequalities (\ref{besi11}) and (\ref{besi44}) hold for all $\gamma>0$.
\end{theorem}


\begin{proof}We first establish inequalities (\ref{besi11})--(\ref{besi55}) and then prove that the constants in inequalities (\ref{100fcp})--(\ref{besi55}), and (\ref{besi11}) when $\gamma=0$, cannot be improved.

(i)  Let us first prove inequality (\ref{besi11}).  The condition $\nu>-\tfrac{1}{2}(n+2)$ ensures that the integral exists. As $\gamma>0$ and $n>-1$, on using the differentiation formula (\ref{diffone}) we have 
\begin{align*}\int_0^x\mathrm{e}^{-\gamma t}t^{\nu}\mathbf{L}_{\nu+n}(t)\,\mathrm{d}t &=\int_0^x\mathrm{e}^{-\gamma t}\frac{1}{t^{n+1}}t^{\nu+n+1}\mathbf{L}_{\nu+n}(t)\,\mathrm{d}t\\
& >\frac{\mathrm{e}^{-\gamma x}}{x^{n+1}}\int_0^xt^{\nu+n+1}\mathbf{L}_{\nu+n}(t)\,\mathrm{d}t =\mathrm{e}^{-\gamma x}x^{\nu}\mathbf{L}_{\nu+n+1}(x),
\end{align*}
since by (\ref{Itend0}) we have $\lim_{x\downarrow 0}x^{\nu+n+1}\mathbf{L}_{\nu+n+1}(x)=0$ if $n>-1$ and $\nu>-\tfrac{1}{2}(n+2)$.

(ii) Using inequality (\ref{Imon}) and then applying (\ref{diffone}) we obtain
\[\int_0^xt^{\nu}\mathbf{L}_{\nu}(t)\,\mathrm{d}t <\int_0^xt^{\nu}\mathbf{L}_{\nu-1}(t)\,\mathrm{d}t=x^{\nu}\mathbf{L}_{\nu}(x).\]

(iii) From the differentiation formula (\ref{diffone}) and the relation (\ref{Iidentity}) we get that
\begin{align*} \frac{\mathrm{d}}{\mathrm{d}t} \big(t^{\nu} \mathbf{L}_{\nu +n+1} (t)\big) &= \frac{\mathrm{d}}{\mathrm{d}t}(t^{-(n+1)} \cdot t^{\nu +n+1} \mathbf{L}_{\nu+n+1} (t)) \\
& = t^{\nu} \mathbf{L}_{\nu +n} (t) -(n+1)t^{\nu -1} \mathbf{L}_{\nu +n+1}(t) \\
& = t^{\nu} \mathbf{L}_{\nu +n} (t) - \frac{n+1}{2(\nu +n+1)} t^{\nu} \mathbf{L}_{\nu+n} (t) + \frac{n+1}{2(\nu +n+1)} t^{\nu} \mathbf{L}_{\nu +n+2} (t) \\
&\quad +(n+1)t^{\nu-1}\cdot\frac{t}{2(\nu+n+1)}\frac{\big(\frac{1}{2}t\big)^{\nu+n+1}}{\sqrt{\pi}\Gamma(\nu+n+\frac{5}{2})} 
\end{align*}
\begin{align*}
& = \frac{2\nu +n+1}{2(\nu +n+1)} t^{\nu} \mathbf{L}_{\nu +n} (t) + \frac{n+1}{2(\nu +n+1)} t^{\nu} \mathbf{L}_{\nu +n+2} (t)\\
&\quad+\frac{n+1}{\nu+n+1}\frac{t^{2\nu+n+1}}{\sqrt{\pi}2^{\nu+n+2}\Gamma(\nu+n+\frac{5}{2})}. 
\end{align*}
Integrating both sides over $(0,x)$, applying the fundamental theorem of calculus and rearranging gives
\begin{align*}\int_0^x t^{\nu} \mathbf{L}_{\nu +n} (t)\,\mathrm{d}t &= \frac{2(\nu +n+1)}{2\nu +n+1} x^{\nu} \mathbf{L}_{\nu +n+1} (x) - \frac{n+1}{2\nu +n+1} \int_0^x t^{\nu} \mathbf{L}_{\nu +n +2} (t)\,\mathrm{d}t \\
&\quad-\frac{2(n+1)}{2\nu+n+1}\int_0^x\frac{t^{2\nu+n+1}}{\sqrt{\pi}2^{\nu+n+2}\Gamma(\nu+n+\frac{5}{2})}\,\mathrm{d}t \\
&= \frac{2(\nu +n+1)}{2\nu +n+1} x^{\nu} \mathbf{L}_{\nu +n+1} (x) - \frac{n+1}{2\nu +n+1} \int_0^x t^{\nu} \mathbf{L}_{\nu +n +2} (t)\,\mathrm{d}t \\
&\quad-\frac{n+1}{2\nu+n+1}\frac{x^{2\nu+n+2}}{\sqrt{\pi}2^{\nu+n+1}(2\nu+n+2)\Gamma(\nu+n+\frac{5}{2})}.
\end{align*}
Applying inequality (\ref{besi11}) with $\gamma=0$ to the integral on the right hand-side of the above expression then yields (\ref{besi22}), as required.

(iv) Let $\nu\geq\frac{1}{2}$.   Then integration by parts and inequality (\ref{100fcp}) gives
\begin{align*} \int_0^x \mathrm{e}^{-\gamma t} t^\nu\mathbf{L}_\nu(t) \,\mathrm{d}t &= \mathrm{e}^{-\gamma x}\int_0^x t^\nu \mathbf{L}_\nu(t)\,\mathrm{d}t + \gamma \int_0^x \mathrm{e}^{-\gamma t}\bigg(\int_0^t  u^\nu\mathbf{L}_\nu(u) \,\mathrm{d}u\bigg) \,\mathrm{d}t \\ &< \mathrm{e}^{-\gamma x}\int_0^x t^\nu \mathbf{L}_\nu(t)\,\mathrm{d}t + \gamma \int_0^x \mathrm{e}^{-\gamma t} t^\nu\mathbf{L}_\nu(t)  \,\mathrm{d}t, 
\end{align*}
whence on rearranging we obtain (\ref{pron}).

(v) Combine parts (iii) and (iv).

(vi) Let $\nu>-\frac{3}{2}$ so that the integral exists. Since $\gamma>0$,
\begin{equation*}\int_0^x \mathrm{e}^{-\gamma t}t^{\nu+1}\mathbf{L}_\nu(t)\,\mathrm{d}t>\mathrm{e}^{-\gamma x}\int_0^x t^{\nu+1}\mathbf{L}_\nu(t)\,\mathrm{d}t=\mathrm{e}^{-\gamma x}x^{\nu+1}\mathbf{L}_{\nu+1}(x).
\end{equation*}

(vii) Consider the function
\begin{equation*}u(x)=\frac{1}{1-\gamma}\mathrm{e}^{-\gamma x}x^{\nu+1}\mathbf{L}_{\nu+1}(x)-\int_0^x\mathrm{e}^{-\gamma t}t^{\nu+1}\mathbf{L}_\nu(t)\,\mathrm{d}t.
\end{equation*}
In order to prove the result, we argue that that $u(x)>0$ for all $x>0$.  Using the differentiation formula (\ref{diffone}) we have that
\begin{align*}u'(x)&=\frac{1}{1-\gamma}\mathrm{e}^{-\gamma x}x^{\nu+1}\big(\mathbf{L}_\nu(x)-\gamma \mathbf{L}_{\nu+1}(x)\big)-\mathrm{e}^{-\gamma x}x^{\nu+1}\mathbf{L}_\nu(x) \\
&=\frac{1}{1-\gamma}\mathrm{e}^{-\gamma x}x^{\nu+1}\big(\mathbf{L}_\nu(x)-\mathbf{L}_{\nu+1}(x)\big)>0,
\end{align*}
where we used (\ref{Imon}) to obtain the inequality.  Also, from (\ref{Itend0}), as $x\downarrow0$,
\begin{align*}u(x)&\sim \frac{1}{1-\gamma}\frac{x^{2\nu+3}}{\sqrt{\pi}2^{\nu+1}\Gamma(\nu+\frac{5}{2})}-\int_0^x \frac{t^{2\nu+2}}{\sqrt{\pi}2^\nu\Gamma(\nu+\frac{3}{2})}\,\mathrm{d}t\\
&=\frac{1}{1-\gamma}\frac{x^{2\nu+3}}{\sqrt{\pi}2^{\nu+1}\Gamma(\nu+\frac{5}{2})}-\frac{x^{2\nu+3}}{\sqrt{\pi}2^\nu(2\nu+3)\Gamma(\nu+\frac{3}{2})} \\
& =\frac{1}{1-\gamma}\frac{x^{2\nu+3}}{\sqrt{\pi}2^{\nu+1}\Gamma(\nu+\frac{5}{2})}-\frac{x^{2\nu+3}}{\sqrt{\pi}2^{\nu+1}\Gamma(\nu+\frac{5}{2})}=\frac{\gamma}{1-\gamma}\frac{x^{2\nu+3}}{\sqrt{\pi}2^{\nu+1}\Gamma(\nu+\frac{5}{2})}>0.
\end{align*}
Thus, we conclude that $u(x)>0$ for all $x>0$, as required.

(viii) We now prove that the constants in inequalities (\ref{100fcp})--(\ref{besi55}) cannot be improved, and that the constant in (\ref{besi11}) is best possible if $\gamma=0$.  We begin by noting that a straightforward asymptotic analysis using the asymptotic formula (\ref{Itendinfinity}) gives that, for $0\leq\gamma<1$ and $\nu>-\frac{1}{2}(n+2)$, 
\begin{equation}\label{intiinf} \int_0^x \mathrm{e}^{-\gamma t}t^\nu \mathbf{L}_{\nu+n}(t)\,\mathrm{d}t\sim \frac{1}{\sqrt{2\pi}(1-\gamma)}x^{\nu+n-\frac{1}{2}}\mathrm{e}^{(1-\gamma)x}, \quad x\rightarrow\infty.
\end{equation}
Let us now prove that the constant in (\ref{100fcp}) is best possible.  From (\ref{intiinf}) and (\ref{Itendinfinity}), we have on the one hand, as $x\rightarrow\infty$,
\begin{equation}\label{eqeq1}\int_0^x t^\nu \mathbf{L}_{\nu}(t)\,\mathrm{d}t\sim  \frac{1}{\sqrt{2\pi}}x^{\nu-\frac{1}{2}}\mathrm{e}^x,
\end{equation}
and on the other,
\begin{equation}\label{eqeq2}x^\nu\mathbf{L}_\nu(x)\sim  \frac{1}{\sqrt{2\pi}}x^{\nu-\frac{1}{2}}\mathrm{e}^x.
\end{equation}
The equivalence between (\ref{eqeq1}) and (\ref{eqeq2}) proves the claim that the constant is best possible.  One can also use this approach to prove that the constants in inequalities (\ref{besi22}), (\ref{pron}), (\ref{besi33}) and (\ref{besi55}), and (\ref{besi11}) when $\gamma=0$, cannot be improved.  

It now remains to prove that the constant in (\ref{besi44}) cannot be improved.  Before doing so, we note that an alternative argument can be used to prove that constant in (\ref{besi22}) is best possible.  From (\ref{Itend0}), we have on the one hand, as $x\downarrow0$,
\begin{equation}\label{form1}\int_0^x t^\nu \mathbf{L}_{\nu+n}(t)\,\mathrm{d}t\sim\int_0^x \frac{t^{2\nu+n+1}}{\sqrt{\pi}2^{\nu+n}\Gamma(\nu+n+\frac{3}{2})}\,\mathrm{d}t=\frac{x^{2\nu+n+2}}{2^{\nu+n}(2\nu+n+2)\Gamma(\nu+n+\frac{3}{2})},
\end{equation}
and on the other,
\begin{align}\label{form2}&\frac{x^\nu}{2\nu+n+1}\bigg(2(\nu+n+1)\mathbf{L}_{\nu+n+1}(x)-(n+1)\mathbf{L}_{\nu+n+3}(x) \nonumber \\
&\quad\quad-\frac{(n+1)x^{\nu+n+2}}{\sqrt{\pi}2^{\nu+n+1}(2\nu+n+2)\Gamma(\nu+n+\frac{5}{2})}\bigg) \nonumber \\
&\quad\sim\frac{x^\nu}{2\nu+n+1}\bigg(\frac{2(\nu+n+1)x^{\nu+n+2}}{\sqrt{\pi}2^{\nu+n+1}\Gamma(\nu+n+\frac{5}{2})}-\frac{(n+1)x^{\nu+n+2}}{\sqrt{\pi}2^{\nu+n+1}(2\nu+n+2)\Gamma(\nu+n+\frac{5}{2})}\bigg) \nonumber \\
&\quad=\frac{(\nu+n+\frac{3}{2})x^{2\nu+n+2}}{\sqrt{\pi}2^{\nu+n}(2\nu+n+2)\Gamma(\nu+n+\frac{5}{2})} =\frac{x^{2\nu+n+2}}{2^{\nu+n}(2\nu+n+2)\Gamma(\nu+n+\frac{3}{2})},
\end{align}
which proves the claim.

Finally, we prove that the constant in the bound (\ref{besi44}) cannot be improved.  Let $M>0$ and define
\begin{equation*}u_M(x)=M\mathrm{e}^{-\gamma x}x^{\nu+1}\mathbf{L}_{\nu+1}(x)-\int_0^x\mathrm{e}^{-\gamma t}t^{\nu+1}\mathbf{L}_\nu(t)\,\mathrm{d}t.
\end{equation*}
From a similar argument to the one used in part (v), we have that, as $x\downarrow0$,
\begin{equation*}u_M(x)\sim (M-1)\frac{x^{2\nu+3}}{\sqrt{\pi}2^{\nu+1}\Gamma(\nu+\frac{5}{2})}.
\end{equation*}
Thus, if $M>1$ then $u_M(x)>0$ in a small positive neighbourhood of the origin, from which we conclude that the constant ($M=1$) in (\ref{besi44}) is best possible.
\end{proof}




We end by noting that we can combine the inequalities of Theorem \ref{tiger} and the integral formula (\ref{besint6}) to obtain lower and upper bounds for a generalized hypergeometric function.  


\begin{corollary}\label{struvebessel}Let $\nu>-\frac{1}{2}$. Then, for all $x>0$,
\begin{align*}&\mathbf{L}_{\nu+1}(x)<\frac{x^{\nu+2}}{\sqrt{\pi}2^{\nu+1}(\nu+1)\Gamma(\nu+\frac{3}{2})}{}_2F_3\bigg(1,\nu+1;\frac{3}{2},\nu+\frac{3}{2},\nu+2;\frac{x^2}{4}\bigg) \nonumber \\
&\quad\quad\quad\quad<\mathbf{L}_{\nu+1}(x)\bigg\{1+\frac{1}{2\nu+1}\bigg(1-\frac{\mathbf{L}_{\nu+3}(x)}{\mathbf{L}_{\nu+1}(x)}\bigg)\bigg\}-\frac{x^{\nu+2}}{\sqrt{\pi}2^{\nu+2}(2\nu+1)(\nu+1)\Gamma(\nu+\frac{5}{2})}.
\end{align*}
\end{corollary}

\begin{proof}Combine the integral formula (\ref{besint6}) and inequalities (\ref{besi11}) and (\ref{besi22}) (with $\gamma=n=0$) of Theorem \ref{tiger}. 
\end{proof}

\begin{remark}We know from Theorem \ref{tiger} that the constants in the double inequality in Corollary \ref{struvebessel} are best possible.  Also, the double inequality is clearly tight in the limit $\nu\rightarrow\infty$ and part (viii) of the proof of Theorem \ref{tiger} tells us that the inequality is tight as $x\rightarrow\infty$.  

To gain further insight into the approximation, we used Mathematica to carry out some numerical results.  Denote by $L_\nu(x)$ and $U_\nu(x)$ the lower and upper bounds in the double inequality and denote by $F_\nu(x)$ the expression involving the generalized hypergeometric function that is bounded by these quantities.  The relative error in approximating $F_\nu(x)$ by $L_\nu(x)$ and $U_\nu(x)$ are given in Tables \ref{table1} and \ref{table2}.  For a given $x$, we observe the relative error in approximating $F_\nu(x)$ by either $L_\nu(x)$ or $U_\nu(x)$ decreases as $\nu$ increases.  We also notice that, for a given $\nu$, the relative error in approximating $F_\nu(x)$ by $L_\nu(x)$ decreases as $x$ increases.  However, from Table \ref{table2} we see that, for a given $\nu$, as $x$ increases the relative error in approximating $F_\nu(x)$ by $U_\nu(x)$ initially increases before decreasing.  This is because, for $\nu>-\frac{1}{2}$, $\lim_{x\downarrow0}\frac{U_\nu(x)}{F_\nu(x)}=1$, and so the relative error in approximating $F_\nu(x)$ by $U_\nu(x)$ is 0 in the limit $x\downarrow0$.  The limit $\lim_{x\downarrow0}\frac{U_\nu(x)}{F_\nu(x)}=1$ follows from combining the formula $F_\nu(x)=x^{-\nu}\int_0^x t^\nu \mathbf{L}_\nu(t)\,\mathrm{d}t$ and the limiting forms (\ref{form1}) and (\ref{form2}) (with $n=0$).

\begin{table}[h]
\begin{center}
\caption{\footnotesize{Relative error in approximating $F_\nu(x)$ by $L_\nu(x)$.}}
\label{table1}
{\scriptsize
\begin{tabular}{|c|rrrrrrr|}
\hline
 \backslashbox{$\nu$}{$x$}      &    0.5 &    5 &    10 &    15 &    25 &    50 & 100   \\
 \hline
$-0.25$ & $0.3975$& $0.2347$ & $0.1114$ & $0.0709$ & $0.0414$ & $0.0203$ & 0.0101  \\
0 & $0.3315$& $0.2099$ & $0.1071$ & $0.0695$ & $0.0409$ & $0.0202$ & 0.0101  \\
2.5 & $0.1251$& $0.1073$ & $0.0773$ & $0.0570$ & $0.0366$ & $0.0192$ & 0.0098  \\
5 & $0.0769$& $0.0715$ & $0.0591$ & $0.0475$ &  $0.0329$  &  0.0182 & 0.0095 \\
7.5 & $0.0555$& $0.0533$ & $0.0472$ & $0.0402$ & $0.0296$ & $0.0173$ & 0.0093  \\ 
10 & $0.0435$& $0.0423$ & 0.0390  & 0.0346    &  0.0268    &  0.0164 & 0.0091 \\  
  \hline
\end{tabular}}
\end{center}
\end{table}
\begin{table}[h]
\begin{center}
\caption{\footnotesize{Relative error in approximating $F_\nu(x)$ by $U_\nu(x)$.}}
\label{table2}
{\scriptsize
\begin{tabular}{|c|rrrrrrr|}
\hline
 \backslashbox{$\nu$}{$x$}      &    0.5 &    5 &    10 &    15 &    25 &    50 & 100   \\
 \hline
$-0.25$ & $0.0087$& $0.4204$ & $0.4288$ & $0.3267$ & $0.2137$ & $0.1134$ & 0.0584  \\
0 & $0.0046$& $0.1781$ & $0.1956$ & $0.1543$ & $0.1034$ & $0.0558$ & 0.0289  \\
2.5 & $0.0001$& $0.0074$ & $0.0142$ & $0.0148$ & $0.0125$ & $0.0080$ & 0.0045 \\
5 & $0.0000$& $0.0015$ & $0.0038$ & $0.0049$ &  $0.0050$  &  0.0037 & 0.0023 \\ 
7.5 & $0.0000$& $0.0005$ & $0.0014$ & $0.0021$ &  $0.0026$  &  0.0022 & 0.0014 \\
10 & $0.0000$& $0.0002$ & 0.0006  & 0.0011    &  0.0015     &  0.0014 & 0.0010 \\  
  \hline
\end{tabular}}
\end{center}
\end{table}

\end{remark}

\subsection*{Acknowledgements}
The author is supported by a Dame Kathleen Ollerenshaw Research Fellowship.

\end{document}